\documentclass[10pt]{amsart}

\usepackage{hyperref}
\usepackage{amsmath}
\usepackage{graphicx}

\usepackage{tikz}
\usepackage{tikz-cd}
\usepackage{tikz-3dplot}
 
\usetikzlibrary{snakes}
\usepackage[all]{xy}
\usepackage{framed}
\usepackage{mathrsfs}
\usepackage{amssymb}
\usepackage{float}

\usetikzlibrary{arrows}

\theoremstyle{plain}
\newtheorem{theorem}{Theorem}[section]

\newtheorem{proposition}[theorem]{Proposition}
\newtheorem{corollary}[theorem]{Corollary}

\theoremstyle{definition}
\newtheorem{definition}[theorem]{Definition}
\newtheorem{example}[theorem]{Example}
\newtheorem{remark}[theorem]{Remark}

\numberwithin{equation}{section}

\newcommand{\CC}{\mathbb{C}}

\newcommand{\QQ}{\mathbb{Q}}
\newcommand{\RR}{\mathbb{R}}
\newcommand{\ZZ}{\mathbb{Z}}
\newcommand{\NN}{\mathbb{N}}

\newcommand{\ts}{\textsc}

\begin{document}

\title[]{GKM theory for orbifold stratified spaces and application to singular toric varieties}

\author[S. Sarkar]{Soumen Sarkar}
\address{Department of Mathematics, Indian Institute of Technology Madras, Chennai 600036, India}
\email{soumensarkar20@gmail.com}

\author[J. Song]{Jongbaek Song}
\address{School of Mathematics, KIAS, 85 Hoegiro Dongdaemun-gu, Seoul 02455, Republic of Korea}
\email{jongbaek.song@gmail.com}

\subjclass[2010]{Primary 14F43, 55N91; Secondary 19L47, 57R85, 57R91}
\keywords{generalized equivariant cohomology, cohomology, K-theory, complex cobordism, toric variety}

\maketitle 
\abstract We study the GKM theory for a equivariant stratified space having orbifold structures in tis successive quotients.  Then, we introduce the notion of an \emph{almost simple polytope}, as well as  a \emph{divisive toric variety} generalizing the concept of a divisive weighted projective space. We employ the GKM theory to compute the generalized equivariant cohomology theories of toric varieties associated to almost simple polytopes and divisive toric varieties. 
\endabstract

\section{Introduction}
A toric variety of complex dimension $n$ is a complex algebraic variety 
with an action of the algebraic torus $(\CC^\ast)^n$ having an open dense orbit. 
It is equipped with a natural action of compact $n$-dimensional torus $T^n$.
The category of toric varieties has been one of the main attractions in algebraic and symplectic geometry from the beginning of 1970s.  
One of the reasons is their rich interaction with different 
fields of mathematics, such as  representation 
theory and combinatorics. For instance, one may get a lattice polytope $P$
from a projective toric variety $X$ by the convexity theorem \cite{Ati-conv, GS} and 
vice versa by Delzant's construction \cite{Del, Gui}. Moreover, the lattice points in $P$
give a weight decomposition of $H^0(X, \mathcal{L})$ as a torus representation, 
where $\mathcal{L}$ is a very ample line bundle over $X$. 

From the topological point of view, such a correspondence leads us to ask how to extract topological invariants for a toric variety from the associated combinatorics, i.e., lattice polytopes or their normal fans. Indeed, there is a rich and vast literature dealing with this question for 
several invariants. For example, we refer to 
\cite{Dan,  DJ, Jur},  \cite{Mor}, \cite{BB}, \cite{BR-MU} for 
non-equivariant cohomology theories,  and \cite{Bag, VV} for equivariant cohomology theories. 
However, most of the computations are focused on the category of smooth toric varieties. 

Now, we change gears to singular toric varieties. Simplicial toric varieties, namely toric varieties having at worst orbifold singularities, may be the mildest class of singular toric varieties. Their ordinary cohomology and Borel equivariant cohomology over rational coefficients behave in a similar manner to smooth toric varieties (see \cite[Section 12.3]{CLS}), while their integral cohomology theories are known only for particular classes under some hypothesis \cite{Ka, Amr, BSS, BNSS}. 

For even worse singular toric varieties, their topological invariants are far away from the situation of smooth toric varieties. For example, their ordinary cohomology may not vanish in odd degrees
in general, which complicates the computation of some spectral sequences
such as Leray--Serre or Atiyah--Hirzebruch spectral sequences.   



In this paper, we introduce the concept of an \emph{orbifold stratified $G$-space} $X$ for some topological group $G$, i.e., $X$ is a finite $G$-CW complex with an equivariant stratification 
\begin{equation}\label{eq_orb_stra}
X_1 \subseteq X_2 \subseteq \cdots \subseteq X_\ell=X
\end{equation}
such that each of the successive quotients $X_j/X_{j-1}$ is homeomorphic to the Thom space of an orbifold $G$-vector bundle and $X_j - X_{j-1}$ is equipped with an effective orbifold structure. We note that the total  space $X$ may have arbitrary singularities. For instance, toric varieties associated to polytopes illustrated in Examples \ref{ex_cone_on_simple_poly}, \ref{ex_GZ_polytope} and \ref{ex_BIP} are singular, but not orbifolds. A relevant concept of an orbifold stratified space is studied in \cite[Definition 1.1]{CLW} in the language of Lie groupoid. 

The main purpose of this paper is to study the generalized GKM theory for the category of orbifold $G$-stratified spaces as in \eqref{eq_orb_stra}. Over this category, we give a concrete description of complex-oriented generalized equivariant cohomologies with rationals, namely, we consider $E_{G}^\ast(-)\otimes_\ZZ \QQ$.  For example,  $E_{G}^\ast$ can be Borel equivariant cohomology $H_{G}^\ast$ or complex equivariant $K$-theory $K_{G}^\ast$ following \cite{segal}. In particular, if the orbifold singularity of $X_j-X_{j-1}$ is trivial for all $j$, we recover the main results of \cite{HHH, HHRW}.

This paper is organized as follows. In Section \ref{sec:GenEqCohom}, 
we discuss the notion of a simple orbifold $G$-bundle and the equivariant Thom isomorphism as a foundation of the GKM theory for orbifold stratified spaces. Then, we introduce the definition of an orbifold $G$-stratification and verify how the generalized GKM theory of \cite{HHH} can be extended to the category of orbifold $G$-stratifications. 


Section \ref{sec:toric_variety} is devoted to a combinatorial characterization 
of toric varieties for which our main results hold. Such a class of toric varieties 
may have arbitrary toric singularities beyond orbifold singularities. Here, we 
bring the idea of \emph{retraction sequence} \cite{BSS} of a simple polytope and extend this to the category of general convex polytopes. This allows us to give an orbifold torus-equivariant stratification on the corresponding toric variety, see Theorem \ref{thm:q-cw_srtucture}. For these toric varieties, we give the GKM theoretic description of generalized cohomology theories in Proposition \ref{prop_GKM-description}.

In Section \ref{sec:PPandAPPLICATIONS}, we summarize the concept of 
a piecewise algebra associated to a fan, which is studied in \cite[Section 4]{HHRW}. Then, we establish Theorem \ref{thm_main_E*_T=PP} describing $E_{T^n}^\ast(X)\otimes \QQ$ 
where $X$ is a singular toric variety discussed in Section \ref{sec:toric_variety} and $T^n$ is the compact torus acting on $X$. 

Finally, generalizing the idea of a divisive weighted projective space, 
we introduce the notion of a \emph{divisive toric variety} in Section \ref{sec:divisive_toric} 
to compute generalized equivariant cohomologies with integers. 
The conclusion is stated in Proposition \ref{prop_divisive_PP}.

We close the introduction with some previous works relevant to the study of this article. The author of \cite{Gon} considers `$\QQ$-filterable spaces' and if they are projective $T$-varieties then they have a stratification similar to \eqref{eq_orb_stra} where $X_j - X_{j-1}$ is a `rational cell' which may not be an orbifold. Under the assumption of `$T$-skeletal', he studies GKM-theory to obtain the Borel equivariant cohomology of those spaces. Nevertheless we are also interested in other generalized equivariant cohomology theories. 

The authors of \cite{HW} studies equivariant $K$-theory of toric varieties associated to `fans with distant singular cones', where they use Mayer--Vietories sequence to show their main theorem \cite[Theorem 7.2]{HW}. There are many singular toric varieties beyond their consideration with stratification as in  \eqref{eq_orb_stra}. For instance, simplicial toric varieties such that all fixed points are singular are excluded from their study.  

We also note that \cite{DKU}, \cite{SU} and \cite{Sar} discuss integral equivariant cohomology theories for (quasi)toric manifolds, toric orbifolds \cite{DJ} and locally standard torus orbifolds \cite{HM}, respectively. 
We emphasize that there are many interesting toric varieties which are not orbifolds, such as the Gelfand--Zetlin toric variety in Example \ref{ex_GZ_toric_PP} or see \cite{KaVi}, whose integral generalized equivariant cohomology theories ($H_T^\ast, K_T^\ast$ and $MU_T^\ast$) can be described by using Proposition \ref{prop_divisive_PP}.

\subsection*{Acknowledgements}
The authors would like to express sincere gratitude to anonymous referees for  helpful comments. The first author would like to thank the international relation office of IIT-Madras and SERB India for MATRICS grant MTR/2018/000963. The second author has been supported by Basic Science Research Program through the National Research Foundation of Korea (NRF) funded by the Ministry of Education (NRF-2018R1D1A1B07048480) and a KIAS Individual Grant (MG076101) at Korea Institute for Advanced Study.


\section{GKM theory for orbifold stratified spaces}
\label{sec:GenEqCohom}
The goal of this section is to apply the GKM theory studied in \cite{HHH} to the category of equivariant stratified $G$-spaces, where certain orbifold structures are involved in their successive quotients. By a $G$-space we mean a finite $G$-CW complex for a topological group $G$. In this paper, we are interested in 
$G$-equivariant cohomology theory $E_G^\ast$ associated to a ring $G$-spectrum $E$ as defined in  \cite[Chapter XIII]{May}. We note that $E^\ast_G(X)$ is a commutative ring together with the structure of $E^\ast_G(pt)$-algebra induced from the equivariant collapsing map $X \to \{pt\}$. In particular, we study $E^\ast_G(X) \otimes \QQ$ for a $G$-space $X$. For simplicity, we use $E^*_G(X)$ in place of $E^*_G(X) \otimes \QQ$. Note that for Borel equivariant cohomology $H_G^\ast$, we have $H_G^\ast(X)\otimes \QQ\cong H_G^\ast(X;\QQ)$.

To establish the structure of stratification of a $G$-space, we begin with a complex $E$-orientable $G$-vector bundle $\xi \colon V \to B$ and a finite group $A$ acting linearly each fiber of $\xi$, which commutes with $G$-action on $V$ and preserves $E$-orientation of $\xi$. (See \cite[Chapter XVI, Definition 9.1]{May} for the definition of $E$-orientation.) Then, one may consider the induced fiber bundle 
$$\xi^A \colon V/A \to B,$$ 
which we call a {\it simple orbifold $G$-bundle}.
The associated disc bundle $D(V) \to B$ and the sphere bundle $S(V) \to B$ of $\xi$ are
invariant under $A$-action, as $A$ acts linearly on each fiber. 
Hence, one can define a \emph{$\mathbf{q}$-disc bundle} 
$D(V/A)(=D(V)/A) \to B$ and a \emph{$\mathbf{q}$-sphere
bundle} $S(V/A)(=S(V)/A) \to B$ in the usual manner, which yields 
the Thom space $ \mathrm{Th}(V/A) := D((V/A)/S(V/A))$ of $\xi^A$ and the map 
$$\tilde{\xi}^A \colon  \mathrm{Th}(V/A) \to B.$$ 
We refer to \cite[Section 4]{PS} and \cite[Section 2]{BNSS} for the notions of a $\mathbf{q}$-disc and a $\mathbf{q}$-sphere.

When the cohomology theory $E_G^\ast$ is Borel equivariant cohomology $H_G^\ast$ or equivariant $K$-theory $K_G^\ast$ (see \cite{segal}), then we have the equivariant Thom isomorphism for a simple orbifold $G$-bundle. 

\begin{proposition}[Thom isomorphism]\label{prop_thom_isomorphism}
Let $\xi^A \colon V/A \to B$ be a simple orbifold $G$-bundle of rank $n$. For cohomology theories $E_G^\ast=H_G^\ast$ or $K_G^\ast$, there exists a cohomology class $\eta_A \in E^n_G( \mathrm{Th}(V/A))$ such that 
$$\cup {\eta_A} \colon E_G^\ast (X) \to E_G^{\ast+n}( \mathrm{Th}(V/A))$$
is an isomorphism.
\end{proposition}
\begin{proof}
Consider the commutative diagram 
\begin{equation*}\label{eq_diag_Borel_fibrations}
\begin{tikzcd}
 \mathrm{Th}(V) \arrow{r}\arrow{d} & EG\times_G  \mathrm{Th}(V) \arrow{r} \arrow{d} & BG\arrow{d}{=}\\
  \mathrm{Th}(V/A) \arrow{r} & EG\times_{G}  \mathrm{Th}(V/A) \arrow{r} &   BG,
\end{tikzcd}
\end{equation*}
where vertical maps are projections induced from the $A$-action and two horizontal compositions are Borel fibrations. Applying Leray--Serre spectral sequence for two Borel fibrations, we get the isomorphisms 
\begin{align}
\label{eq_iso-1}H^\ast_G(\mathrm{Th}(V))&\cong H^\ast(\mathrm{Th}(V))\otimes_\QQ H^\ast(BG)\\
\label{eq_iso-2}H^\ast_{G}(\mathrm{Th}(V/A)) &\cong  H^\ast(\mathrm{Th}(V/A))\otimes_\QQ H^\ast(BG).
\end{align}
We note that $H^\ast(V) \cong H^\ast(V/A)$ and $H^\ast(V_0) \cong H^\ast(V_0/A)$ with rational coefficients, where $V_0$ denotes the complement of the zero section. Now, the long exact sequence of the pair $(V, V_0)$ together with the Five Lemma shows that $H^\ast(\mathrm{Th}(V))\cong H^\ast(\mathrm{Th}(V/A))$. We refer to \cite[Section 5.1]{PS} for more details. Hence, the left-hand sides of \eqref{eq_iso-1}, \eqref{eq_iso-2} and $H^\ast_G(\mathrm{Th}(V/A))$ are isomorphic.

Consider the commutative diagram
\begin{equation*}
\begin{tikzcd} 
V \arrow{d}{\xi} \arrow{r}{\rho} &V/A\arrow{d}{\xi^A} \\
B \arrow{r}{=} & B,
\end{tikzcd}
\end{equation*}
where $\rho$ is the canonical projection given by the action of $A$ on $V$. 
This induces the diagram
\begin{equation*}
\begin{tikzcd}
H^{\ast+n}_G(\mathrm{Th}(V)) & H_G^{\ast+n}(\mathrm{Th}(V/A)) \arrow{l}[swap]{\rho^\ast} \\
H^{\ast}_G(B)\arrow{u}{\cup \eta} & H^\ast_G(B)\arrow{l}{=}\arrow{u},
\end{tikzcd}
\end{equation*}
where the left vertical map $\cup \eta$ is the equivariant Thom isomorphism (see \cite[Theorem 9.2, Chapter XVI]{May}) given by the cup product of the equivariant Thom class $\eta\in H^\alpha_G(\mathrm{Th}(V))$. Therefore, the right vertical map is an isomorphism given by the cup product of the pull back $\eta_A:= (\rho^\ast)^{-1}(\eta)\in H^n_G(\mathrm{Th}(V/A))$. 

Next, to show the claim for the equivariant $K$-theory, we consider the composition $ch \circ \psi$, where 
$ch$ denotes the \emph{Chern character} from Borel equivariant $K$-theory to equivariant cohomology and $\psi$ is the canonical monomorphism from $K_G(\mathrm{Th}(V))$ to $K(EG\times_G \mathrm{Th}(V))$ defined by assigning the vector bundle $EG\times_G \xi$ to each $G$-equivariant vector bundle $\xi$, we refer to \cite{AS}. Applying the same composition to $\mathrm{Th}(V/A)$ and the usual Thom isomorphism theorem for a genuine vector bundle, 
we have the commutative diagram:
\begin{equation*}
\begin{tikzcd}[column sep=small]
K_G^{\ast}(B) \arrow{r}{\cong} &
K_G^{\ast+n}( \mathrm{Th}(V)) \arrow{r}{\psi} & 
K^{\ast+n}(EG\times_G \mathrm{Th}(V)) \arrow{r}{ch}& 
H^{\ast+n}_G(\mathrm{Th}(V))\\
K_G^{\ast}(B) \arrow{r}{\phi}\arrow{u}{=}& 
K_G^{\ast+n}( \mathrm{Th}(V/A)) \arrow{r}{\psi_A}\arrow{u}{f^\ast}& 
K^{\ast+n}(EG\times_G \mathrm{Th}(V/A)) \arrow{r}{ch} \arrow{u}& 
H^{\ast+n}_G(\mathrm{Th}(V/A)),\arrow{u}{\cong}
\end{tikzcd}
\end{equation*}
where we claim that $\phi$ is an isomorphism. Indeed, Chern characters are injective as we are working with rationals. Hence, the surjectivity and the injectivity of $f^\ast$ follow from the commutativity of the left most square and right two squares of the diagram, respectively. 
Therefore, we have $K_G^{\ast}(B) \cong K_G^{\ast+n}( \mathrm{Th}(V/A))$, induced from the Thom isomorphism $ K_G^{\ast}(B) \cong
K_G^{\ast+n}( \mathrm{Th}(V))$ for a genuine vector bundle $V\to B$. 
\end{proof}

Throughout this paper, we consider a simple orbifold $G$-bundle with rank $n$ over a topological space $B$ and assume that a complex-oriented $G$-equivariant cohomology theory $E_G^\ast$ has the Thom isomorphism 
\begin{equation*}\label{eq_Thom_iso}
E_G^\ast (B) \cong E_G^{\ast+n}( \mathrm{Th}(V/A))
\end{equation*}
for a simple orbifold $G$-bundle given by the cup product of the equivariant Thom class $\eta_A \in E_G^{\ast}( \mathrm{Th}(V/A))$. Then, one can define the equivariant Euler class $e_G(\xi^A) \in E^{\ast}_G(B)$ of a simple orbifold $G$-bundle $\xi^A$ by the restriction of $\eta_A$ to the zero section. 

Now we consider an equivariant stratification 
\begin{equation}\label{eq_g-tratification}
\{pt\}=X_1 \subseteq X_2 \subseteq  \cdots\subseteq X_{\ell-1} \subseteq X_\ell = X
\end{equation}
of a $G$-space $X$ such that each of the  successive quotients
$X_j/X_{j-1}$ is homeomorphic to the Thom space $\mathrm{Th}(V_j/A_j)$
of a simple orbifold $G$-bundle $\xi^{A_j} \colon V_j/A_j \to B_j$. 
Therefore, $X$ can be built from $X_1$ inductively by attaching
$\mathbf{q}$-disc bundles $D(V_j/A_j)$ to $X_{j-1}$ via some $G$-equivariant maps
$$\phi_j : S(V_j/A_j) \to X_{j-1},$$
which gives us a cofibrations 
\begin{equation}\label{eq_cofib_thom}
X_{j-1} \to X_j \xrightarrow{q} \mathrm{Th}(V_j/A_j)
\end{equation} 
for $j=2, \dots, \ell$. Now, one gets the following proposition by the induction on the stratification \eqref{eq_g-tratification}.

\begin{proposition}\label{prop_injective}
Let $X$ be an orbifold stratified $G$-space as in \eqref{eq_g-tratification}.  If each equivariant Euler class
$e_G(\xi^{A_j}) \in E^*_G(B_j)$ of the associated simple orbifold $G$-bundle $\xi^{A_j}$ is not a zero divisor, then the map 
\begin{equation}\label{eq_injectivity}
\iota^* \colon E^*_G(X) \to \prod_{j} E^*_G(B_j).
\end{equation}
induced from the inclusion $\iota \colon \bigsqcup B_j \hookrightarrow X$ is injective.
\end{proposition}
\begin{proof}
Essentially, the argument is similar to the proof of \cite[Theorem 2.3]{HHH}. Here we briefly make the foundation for the orbifold stratification \eqref{eq_g-tratification}. If the stratification \eqref{eq_g-tratification} has length 1, then $X_1 = B_1$ is a point. Therefore, \eqref{eq_injectivity} is an isomorphism.

Let the stratification \eqref{eq_g-tratification} have length $\ell$ and  assume that \eqref{eq_injectivity} is an injective map for any  stratification \eqref{eq_g-tratification} of length less than $\ell$. By the assumption on the stratification, we have the cofiber sequence \eqref{eq_cofib_thom} for each $j<\ell$. Since equivariant Euler classes are not zero divisors and a cohomology theory $E^\ast_G$ is considered to have the equivariant Thom isomorphism, we get a short exact sequence
$$0 \to E^*_G(\mathrm{Th}(V_j/A_j)) \xrightarrow{q^*} E^*_G(X_j) \to E^*_G(X_{j-1}) \to 0.$$
Hence, we have a commutative diagram
\begin{equation*}
\begin{tikzcd}
0 \arrow{r}& E^*_G(\mathrm{Th}(V_j/A_j)) \arrow{r}{q^*} \arrow{d}&E^*_G(X_j) \arrow{r}\arrow{d}& E^*_G(X_{j-1})\arrow{d} \arrow{r}& 0\\
0 \arrow{r} & E^\ast_G(B_j) \arrow{r}&\displaystyle \prod_{i\leq j}E_G^\ast(B_i) \arrow{r} & \displaystyle\prod_{i<j}E_G^\ast(B_i) \arrow{r} & 0,
\end{tikzcd}
\end{equation*}
where the left vertical map is injective, as $e_G(\xi^{A_j})$ is not a zero divisor. The right vertical map is also injective by the induction hypothesis. Now the Five Lemma completes the proof.
\end{proof}

To describe the image of $\iota^\ast$ in \eqref{eq_injectivity}, we set up the following assumptions on a stratified $G$-space $X$. 
\begin{enumerate}
\item[(\textbf{A}1)] Simple orbifold bundles $\xi^{A_j} \colon V_j/A_j \to B_j$ for $j=2, \dots \ell$ are $E$-orientable and have decompositions 
\begin{equation}\label{eq_orbibundle_decomp}
(\xi^{A_j} \colon V_j/A_j \to B_j) \cong \bigoplus_{s < j} (\xi^{A_{j,s}} 
\colon V_{j,s }/A_{j,s} \to B_j)
\end{equation}
into simple orbifold bundles $\xi^{A_{j,s}}$, 
possibly $V_{j,s}$ can be trivial. We note that \eqref{eq_orbibundle_decomp} is inherited from the  decomposition 
$$(\xi\colon V_j \to B_j) \cong \bigoplus_{s < j} (\xi \colon V_{j,s} \to B_j)$$ 
of vector bundles and $A_{j,s}$'s are the quotients of $A_j$ by non-effective kernels. 
\item[(\textbf{A}2)] The restriction $\phi_j|_{S(V_{j,s}/A_{j,s})}$ of the attaching map $\phi_j \colon S(V_j/A_j) \to X_{j-1}$ to $S(V_{j,s}/A_{j,s})$ satisfies 
$$\phi_j|_{S(V_{j,s}/A_{j,s})} = f_{j,s} \circ \xi^{A_{j,s}}$$ 
for some $G$-equivariant map $f_{j,s} \colon B_j \to B_{s}$,
identifying $B_{s}$ with its image in $X_{j-1}$ for each $s < j$. 
\item[(\textbf{A}3)] The equivariant Euler classes $e_G(\xi^{A_{j,s}})$ are not zero divisors
and pairwise relatively prime in $E^{\ast}_G(B_j)$. 
\end{enumerate} 

We remark that the $G$-invariant stratifications with trivial  $A_j$'s are
studied in \cite{HHH}. Nevertheless, under the above assumptions
on a $G$-space $X$ with the property as in \eqref{eq_g-tratification}, 
one may obtain the following proposition. 

\begin{proposition}\label{prop_image}
Let $X$ be an orbifold stratified $G$-space as in
\eqref{eq_g-tratification} and satisfy assumptions {\rm (\textbf{A}1)} to {\rm (\textbf{A}3)}. Then the image of $\iota^* \colon E^*_G(X) \to \prod_{j} E^*_G(B_j)$ is
$$\Gamma_X := \Big \{ (x_j) \in \prod_{1\leq j\leq \ell} E^*_{G}(B_j) ~ \Big{|} ~ e_G(\xi^{A_{j,s}})~|~ x_j - f^*_{j,s}(x_{s}) ~\text{for}~ s < j\Big \}.$$
\end{proposition}

\begin{proof}
The proof can be obtained from proof of \cite[Theorem 3.1]{HHH} by replacing genuine $G$-vector bundles $V_{j,s}\to B_{j}$ and their equivariant Euler classes into simple orbifold $G$-bundles $\xi^{A_{j,s}}\colon V_{j,s}/A_{j,s}\to B_j$ and corresponding equivariant Euler classes $e_G(\xi^{A_{j,s}})$. For the reader's convenience, we briefly outline the argument here. 

The proof goes by the induction on the filtration. For $X_1=\{pt\}$, the claim holds as $e_G(\xi^{A_1})=1$. Now, we suppose that the claim holds for $X_{j-1}$. From assumptions (\textbf{A}1) -- (\textbf{A}3) on the filtration, we have subspaces $X_{j,s}$ of $X$ for $1\leq s < j$ such that $$ X_{j,s} = B_s \cup_{f_{j,s} \circ \xi^{A_{j,s}}} D(V_{j,s}/A_{j,s}).$$
Then, following the proof of \cite[Lemma 3.5]{HHH} and Proposition \ref{prop_thom_isomorphism}, one can show $e_G(\xi^{A_{j,s}})$ divides $x_j - f_{j,s}^{\ast}(x_s)$ where $x_j$ and $x_s$ are pull-backs of a class $x \in E^*_G(X)$ under $B_j \hookrightarrow X$ and $B_s \hookrightarrow X$, respectively. 

Now we consider the natural restriction map $\gamma_j \colon \Gamma_j \to \Gamma_{j-1}$, where
$$\Gamma_i := \Big \{ (x_k) \in \displaystyle \prod_{k \leq i} E_G^*(B_k) ~\Big{|} ~ e_G(\xi^{A_{k,r}}) | x_k - f_{k,r}^*(x_r) ~\text{ for } r<k  \Big{\}}.$$
Then (\textbf{A}3) verifies that $\ker(\gamma_j) \cong e_G(\xi^{A_j}) E^*_G(B_j).$ 
Moreover, one can obtain a commutative diagram 
\begin{equation}\label{eq_comm_diag_GKM}
\begin{tikzcd}
0\arrow{r}&
E_G^\ast(X_j, X_{j-1}) \arrow{r}\arrow{d}{\cong} & 
E_G^\ast(X_j) \arrow{r}\arrow{d} & 
E_G^\ast(X_{j-1}) \arrow{r}\arrow{d}{\cong}& 
0 \\
0 \arrow{r} & 
\ker (\gamma_j) \arrow{r}& 
\Gamma_j \arrow{r} & 
\Gamma_{j-1} \arrow{r} & 
0,
\end{tikzcd}
\end{equation}
where the exposition about the validity of \eqref{eq_comm_diag_GKM} is given in \cite{HHH}. Finally, one can complete the proof by the Five Lemma.
%
\end{proof}

\begin{remark}
If all $A_{j}$'s associated to the stratification \eqref{eq_g-tratification} are trivial, then ${\rm Th}(V_j/A_j)={\rm Th}(V_j)$ which is the Thom space for a genuine vector bundle for each $j=1, \dots, \ell$.  A class of examples satisfying this condition will be discussed in Section \ref{sec:divisive_toric}. 
With this assumption, Proposition \ref{prop_injective} and Proposition \ref{prop_image} agree with the first part of Theorem 2.3 and Theorem 3.1 in \cite{HHH}, respectively.  
In this case, the cohomology theory $E_G^\ast$ can also be the complex cobordism $MU_G^\ast$. 
\end{remark}


\section{Toric varieties over almost simple polytopes}\label{sec:toric_variety}
In this section, we give a combinatorial characterization of toric varieties 
which is essential for the main results of this paper. 
Let $\Sigma$ be a full dimensional rational polytopal fan in $\RR^n$
and $P$ the lattice polytope whose normal fan is $\Sigma$. 
The corresponding toric variety $X_\Sigma$ 
is equipped with an action of compact torus $T^n \subset (\CC^\ast)^n$.
Here, we identify $\mathbb{R}^n$ with the Lie algebra of $T^n$.  

Following the result of \cite{Jur2} (we also refer to \cite[Theorem 12.2.5]{CLS}), 
there is a $T^n$-equivariant homeomorphism 
\begin{equation*}\label{eq_equiv_hoemo}
f\colon X_\Sigma \xrightarrow{\cong} (T^n\times P) /_\sim,
\end{equation*}
where $(t, p)\sim (s,q)$ whenever $p=q$ and $t^{-1}s$ is an element of
the subtorus $T_{F(p)}\subseteq T^n$ whose Lie algebra is generated by the outward 
normal vectors of the codimension-$1$ faces of $P$ which contain $p$ if $p$ is not in 
the interior of $P$.  When $p$ is in the interior of $P$, we consider 
$T_{F(p)}$ to be trivial.   

Here, we notice that $T^n$-action on  $(T^n\times P) /_\sim$ is 
induced from the multiplication on the first factor of $T^n\times P$ 
and the corresponding orbit map, 
\begin{equation}\label{eq_orbit_map}
\pi \colon (T^n\times P) /_\sim ~~ \to P,
\end{equation}
is given by $[t, p]_\sim \mapsto p$, where $[t,p]_\sim$ denotes the 
equivalence class of $(t,p)$. 
Therefore, 
the topology of a toric variety can be studied by the 
combinatorics of the orbit space $P$ and its geometric data, namely,
outward normal vectors of codimension-$1$ faces of $P$. 

We discuss the combinatorics of $P$ in Subsection 
\ref{subsec_ret_of_conv_polytope} and study some topological 
information of $X_P$ obtained from the geometry of $P$ 
in Subsection \ref{subsec_torus_eq_stratification}.

\subsection{Retraction sequence of a convex polytope}
\label{subsec_ret_of_conv_polytope}
The goal of this subsection is to introduce a combinatorial characterization of certain convex polytopes which was initiated in \cite{BSS}. 

Let $P$ be a convex polytope of dimension $n$. Regarding $P$ as a polytopal complex \cite[Definition 5.1]{Zie}, i.e., $P$ is the set of all its faces, we consider a finite sequence of triples 
\begin{equation*}\label{eq_ret_seq}
(P_1, Q_1, v_1) \to (P_2, Q_2, v_2)\to \cdots,
\end{equation*}
where $P_j$ is a polytopal subcomplex of $P$, $Q_j$ is a face of $P_j$ and $v_j$ is a vertex of $Q_j$, which are defined inductively as follows.

We set the initial term $(P_1, Q_1, v_1)$ such that $P_1=P$, $v_1$ is a vertex of $P_1$ having a neighborhood homeomorphic to $\RR^{n}_\geq$ as manifold with corners and $Q_1=P_1$ as an element of polytopal complex $P_1$. Given $(P_j, Q_j, v_j)$, the next term $(P_{j+1}, Q_{j+1}, v_{j+1})$ is defined by 
setting 
$$P_{j+1}=\bigcup \{Q\in P_j\mid v_j\notin V(Q)\},$$
where $V(Q)$ is the set of vertices of $Q$.
Next we choose a vertex $v_{j+1}$ of $P_{j+1}$ such that $v_{j+1}$ has a neighborhood  homeomorphic to $\RR^d_\geq$ as manifold with corners for some $1\leq d \leq n$. We call $v_{j+1}$ a \emph{free vertex} of $P_{j+1}$. 
A face $Q_{j+1}$ is defined to be the unique maximal face of $P_{j+1}$ containing $v_{j+1}$. 
Note that a free vertex may not exist in general, see Remark \ref{rmk_non-ret-polytope}.  Hence, we proceed to define a sequence if a free vertex exists. A sequence as defined above is called a \emph{retraction sequence} of $P$ if the sequence ends up with $(P_\ell, Q_\ell, v_\ell)$ such that $P_\ell=Q_\ell=v_\ell$ for some vertex $v_\ell$ of $P$, where $\ell$ denotes the cardinality of $V(P)$.
 
\begin{definition}
A convex polytope is called \emph{almost simple} if it admits 
at least one retraction sequence. 
\end{definition}

A simple convex polytope is almost simple. Indeed, 
for a simple convex polytope $P$ in $\RR^n$, a height function 
$\phi \colon \RR^n \to \RR$ which is \emph{generic} in the sense that each 
vertex of $P$ has different height defines a retraction sequence of $P$. We refer 
to \cite[Proposition 2.3]{BSS} for the details. Notice that not every retraction 
sequence can be obtained from a height function. Below, we introduce several examples 
of non-simple polytopes which are almost simple.

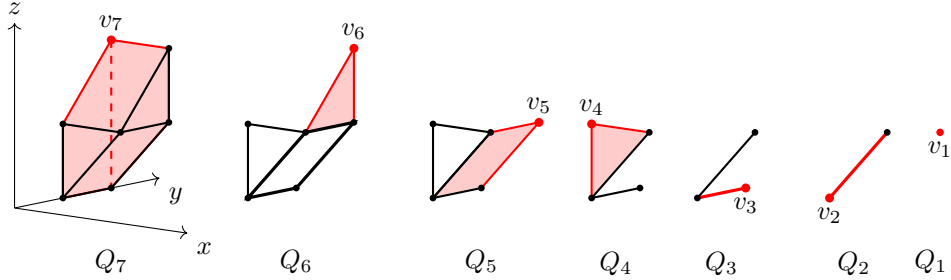
\begin{figure}
\tdplotsetmaincoords{80}{40}
\begin{tikzpicture}[tdplot_main_coords, scale=0.9, yscale=0.9]
\draw[->] (0,0,0) -- (3,0,0) node[anchor=north west]{$x$};
\draw[->] (0,0,0) -- (0,3,0) node[anchor=north west]{$y$};
\draw[->] (0,0,0) -- (0,0,2.5) node[anchor=south]{$z$};
\draw[thick, red, fill=red!20] (0,1,0)--(0,2,0)--(1,2,1)--(1,2,2)--(0,2,2)--(0,1,1)--cycle;
\draw[thick,red, dashed] (0,2,2)--(0,2,0);
\draw[thick] (0,1,1)--(0,1,0)--(0,2,0)--(1,2,1)--(1,2,2);
\draw[thick] (1,1,1)--(0,1,0);
\draw[thick] (1,1,1)--(0,1,1);
\draw[thick] (1,1,1)--(1,2,1);
\draw[thick] (1,1,1)--(1,2,2);
\node[fill=red, circle, inner sep=1.2pt] at (0,2,2) {};
\node[above] at (0,2,2) {\scriptsize$v_1$};

\node[fill, circle, inner sep=.9pt] at (0,1,0) {};
\node[fill, circle, inner sep=.9pt] at (0,1,1) {};
\node[fill, circle, inner sep=.9pt] at (1,1,1) {};
\node[fill, circle, inner sep=.9pt] at (1,2,1) {};
\node[fill, circle, inner sep=.9pt] at (1,2,2) {};
\node[fill, circle, inner sep=.9pt] at (0,2,0) {};
\node at (0,2,-1) {\scriptsize$P_1$};

\begin{scope}[xshift=70]
\draw[thick, red, fill=red!20] (1,1,1)--(1,2,2)--(1,2,1);
\draw[thick] (1,1,1)--(0,1,1)--(0,1,0)--cycle;
\draw[very thick] (1,1,1)--(1,2,1)--(0,2,0)--(0,1,0)--cycle;
\node[fill=red, circle, inner sep=1.2pt] at (1,2,2) {};
\node[above] at (1,2,2) {\scriptsize$v_2$};

\node[fill, circle, inner sep=.9pt] at (0,1,0) {};
\node[fill, circle, inner sep=.9pt] at (0,1,1) {};
\node[fill, circle, inner sep=.9pt] at (1,1,1) {};
\node[fill, circle, inner sep=.9pt] at (1,2,1) {};
\node[fill, circle, inner sep=.9pt] at (0,2,0) {};
\node at (0,2,-1) {\scriptsize$P_2$};
\end{scope}

\begin{scope}[xshift=140]
\draw[thick, white, fill=red!20] (1,1,1)--(1,2,1)--(0,2,0)--(0,1,0)--cycle;
\draw[thick] (1,1,1)--(0,1,0)--(0,2,0);
\draw[thick, red] (1,1,1)--(1,2,1)--(0,2,0);
\draw[thick] (1,1,1)--(0,1,1)--(0,1,0)--cycle;

\node[fill=red, circle, inner sep=1.2pt] at (1,2,1) {};
\node[above] at (1,2,1) {\scriptsize$v_3$};
\node[fill, circle, inner sep=.9pt] at (0,1,0) {};
\node[fill, circle, inner sep=.9pt] at (0,1,1) {};
\node[fill, circle, inner sep=.9pt] at (1,1,1) {};
\node[fill, circle, inner sep=.9pt] at (0,2,0) {};
\node at (0,2,-1) {\scriptsize$P_3$};
\end{scope}

\begin{scope}[xshift=200]
\draw[thick, white, fill=red!20] (1,1,1)--(0,1,1)--(0,1,0)--cycle;
\draw[thick, red] (0,1,1)--(0,1,0);
\draw[thick, red] (0,1,1)--(1,1,1);
\draw[thick] (1,1,1)--(0,1,0);
\draw[thick] (0,2,0)--(0,1,0)--cycle;
\node[fill, circle, inner sep=.9pt] at (1,1,1) {};
\node[fill, circle, inner sep=.9pt] at (0,1,0) {};
\node[fill=red, circle, inner sep=1.2pt] at (0,1,1) {};
\node[above] at (0,1,1) {\scriptsize$v_4$};

\node[fill, circle, inner sep=.9pt] at (0,2,0) {};
\node at (0,1.5,-0.95) {\scriptsize$P_4$};
\end{scope}

\begin{scope}[xshift=240]
\draw[thick] (1,1,1)--(0,1,0);
\draw[very thick, red] (0,2,0)--(0,1,0);
\node[fill=red, circle, inner sep=1.2pt] at (0,2,0) {};
\node[below] at (0,2,0) {\scriptsize$v_5$};

\node[fill, circle, inner sep=.9pt] at (0,1,0) {};
\node[fill, circle, inner sep=.9pt] at (1,1,1) {};
\node at  (0,1.5,-0.95) {\scriptsize$P_5$};
\end{scope}

\begin{scope}[xshift=290]
\draw[very thick, red] (1,1,1)--(0,1,0);
\node[fill, circle, inner sep=0.9pt] at (1,1,1) {};
\node[fill=red, circle, inner sep=1.2pt] at (0,1,0) {};
\node[below] at (0,1,0) {\scriptsize$v_6$};
\node at  (0,1.5,-0.95) {\scriptsize$P_6$};
\end{scope}

\begin{scope}[xshift=310]
\node[fill=red, circle, inner sep=1pt] at (1,1,1) {};
\node[below] at (1,1,1) {\scriptsize$v_7$};
\node at (0,2,-1) {\scriptsize$P_7$};
\end{scope}
\end{tikzpicture}
\caption{A retraction sequence of $3$-dimensional Gelfand--Zetlin polytope.}
\label{fig_ret_seq_of_GZ_polytope}
\end{figure}

\begin{example}\label{ex_cone_on_simple_poly}
A cone $C(P)$ on a simple polytope $P$ has a retraction sequence. Indeed, 
let 
$$
(P_1, Q_1, v_1) \to \cdots \to (P_\ell, Q_\ell, v_\ell)
$$
be a retraction sequence of $P$. Then, 
$$(C(P_1), C(Q_1), v_1) \to \cdots \to (C(P_\ell), C(Q_\ell), v_\ell) \to (\ast, \ast, \ast)$$
is a retraction 
sequence for $C(P)$, where $\ast$ is the apex of $C(P)$. 
Notice that $C(P)$ is not a simple polytope unless 
$P$ is a simplex. 
\end{example}

\begin{example}\label{ex_GZ_polytope}
Let $P$ be a $3$-dimensional polytope given by the system of inequalities described as follows:
\begin{figure}[H]
\begin{tikzpicture}[scale=0.6]
\node at (0,0) {$0$};
\node at (2,0) {$1$};
\node at (4,0) {$2$};

\node[rotate=-45] at (1/2, -1/2) {$\leq$};
\node[rotate=45] at (3/2, -1/2) {$\leq$};
\node[rotate=-45] at (5/2, -1/2) {$\leq$};
\node[rotate=45] at (7/2, -1/2) {$\leq$};

\node at (1,-1) {$x$};
\node at (3,-1) {$y$};

\node[rotate=-45] at (3/2, -3/2) {$\leq$};
\node[rotate=+45] at (5/2, -3/2) {$\leq$};

\node at (2,-2) {$z$};

\end{tikzpicture}
\end{figure}

\noindent It is a $3$-dimensional example of \emph{Gelfand--Zetlin} polytopes which plays an important role particularly in the algebro-geometric study of flag varieties. See Figure \ref{fig_ret_seq_of_GZ_polytope} for a pictorial description of a retraction sequence of $P$. One can also construct different retraction sequences beginning with other vertices except for $v_7$. 
\end{example}

\begin{example}\label{ex_BIP}
Some retraction sequences for $3$-dimensional Bruhat interval polytopes \cite{TW} are illustrated in \cite[Figures 25, 27]{LM}, which are not simple polytopes. 
\end{example}

\begin{remark}\label{rmk_non-ret-polytope}
Not every convex polytope 
has a retraction sequence, for instance the convex hull of 
$\{\pm (1,0,0), \pm(0,1,0), \pm(0,0,1)\}$. It is a $3$-dimensional 
convex polytope with $6$ vertices and each of the vertices does not have 
any neighborhood homeomorphic to $\RR^3_\geq$. 
\end{remark}

In terms of toric varieties, above examples shows that the category of 
toric varieties over almost simple polytopes is strictly larger than 
the category of all simplicial toric varieties. 

\subsection{Torus-equivariant stratifications}
\label{subsec_torus_eq_stratification}
From now on, we consider a toric variety $X$ whose orbit space, 
via the orbit map  $\eta \colon X \to P$ defined in \eqref{eq_orbit_map},
is an almost simple polytope $P$.  Such a toric variety have the following 
property, which is one of the main observations in this paper. 

\begin{theorem}\label{thm:q-cw_srtucture}
A retraction sequence 
$(P_1 , Q_1, v_1) \to \cdots \to (P_\ell , Q_\ell, v_\ell)$ of $P$ yields a $T^n$-equivariant stratification of $X$ 
\begin{equation*}
X_1 \subseteq X_2 \subseteq \cdots \subseteq X_{\ell}=X
\end{equation*}
such that the quotient $X_{j}/X_{j-1}$ 
is homeomorphic to the Thom space $\mathrm{Th}(\xi^{A_j})$
of the simple orbifold $T^n$-bundle 
\begin{equation}\label{eq_orb_vector_bdl_over_pt}
\xi^{A_j} \colon \CC^{k_j}/A_j \to \eta^{-1}(v_{\ell-j+1}),
\end{equation}
for some $k_j \in \NN$ and finite abelian group $A_j$, where $v_{i}\in V(P)$ denotes the
free vertex of $P_{i}$ to define $P_{i+1}$, for $i =1, \dots, \ell-1$. 
\end{theorem}
\begin{proof}
We define $X_j:=\pi^{-1}(P_{\ell-{j}+1})$ for $j=1, \dots, \ell$. 
Then, $X_{j+1}\subset X_{j}$ as $P_{j} \supset P_{j+1}$, see the  diagram
\begin{equation}\label{eq_T_equiv_stratification}
\begin{tikzcd}[column sep=tiny]
X_1 \subseteq  \cdots \subseteq ~~ X_{j-1}~~ \subseteq ~~ X_j ~~ \subseteq  ~~X_{j+1}~~ \subseteq  \cdots \subseteq  X_\ell \arrow{d}{\pi} \arrow[shift left=9ex]{d}{\pi}\arrow[shift right=9ex]{d}{\pi}\arrow[shift left=22ex]{d}{\pi}\arrow[shift right=22ex]{d}{\pi}\\
P_\ell \subseteq  \cdots \subseteq P_{\ell-j+2} \subseteq P_{\ell-j+1} \subseteq  P_{\ell-j} \subseteq  \cdots \subseteq  P_1. 
\end{tikzcd}
\end{equation}
Since $\pi$ is the orbit map with respect to $T^n$-action, \eqref{eq_T_equiv_stratification} is $T^n$-equivariant. 

To prove the second assertion, we consider the unique maximal face $Q_{\ell-j+1}$ of $P_{\ell-j+1}$
which contains the free vertex $v_{\ell-j+1}$ and denote by $U_{\ell-j+1}\subset Q_{\ell-j+1}$ the union of all relative interiors of faces in $Q_{\ell-j+1}$ containing $v_{\ell-j+1}$. For instance, the colored 
faces in Figure \ref{fig_ret_seq_of_GZ_polytope} are $U_{\ell-j+1}$ for $\ell=7$ and $2\leq j\leq 7$. 
Then, one can see from the property of a free vertex that $U_{\ell-j+1}$ is 
homeomorphic to $\RR^{k_j}_\geq$ as manifolds with corners, where $k_j:=\dim U_{\ell-j+1}=\dim Q_{\ell-j+1}$. Also, we note that $$X_{j}-X_{j-1}=\eta^{-1}(U_{\ell-j+1}).$$

Let $\RR(Q_{\ell-j+1})$ be the subspace of $\RR^n$ 
generated by the normal vectors of facets of $P$ intersecting $Q_{\ell-j+1}$. Notice that $\RR(Q_{\ell-k_j+1})$ is of dimension $n-k_j$ as $Q_{\ell-j+1}$ is $k_j$-dimensional face of $P$. 
Since $v_{\ell-j+1}$ is a free vertex of $P_{\ell-j+1}$, there are $k_j$-many facets, say 
$F_1, \dots, F_{k_j}$, such that $v_{\ell-j+1} = \bigcap_{i=1}^{k_j}(Q_{\ell-j+1} \cap F_i)$. 
Consider the projection 
\begin{equation}\label{eq_ground_lattice}
\ZZ^n \to \ZZ^n/(\ZZ^n \cap  \RR(Q_{\ell-j+1}))\cong \ZZ^{k_j}
\end{equation}
and the images $\mu_1, \dots, \mu_{k_j}$ of primitive outward normal vectors 
of $F_1, \dots, F_{k_j}$ via \eqref{eq_ground_lattice}, respectively. 


Now, the result of \cite[Proposition 4.4]{BNSS} shows  
\begin{equation}\label{eq_difference_of_strata_q-disc}
 \pi^{-1}(U_{\ell-j+1}) \cong D^{2k_j}/ A_j,
\end{equation}
where 
\begin{equation}\label{eq_finite_gp_in_filtration}
A_j=\ker (\exp \left[ \mu_{1} \mid \dots \mid \mu_{k_j}\right] 
\colon T^{k_j} \to T^{k_j}).
\end{equation}

Here, one can regard the space $\eqref{eq_difference_of_strata_q-disc}$ 
as a $\mathbf{q}$-disc bundle of a simple orbifold $T^{k_j}$-bundle 
\begin{equation}\label{eq_T_bundle}
\xi^{A_j} \colon \CC^{k_{j}} /A_j \to \pi^{-1}(v_{\ell-j+1}),
\end{equation}
where the standard $T^{k_{j}}$-action on $\CC^{k_{j}}$ 
induces an action on $\CC^{k_{j}}/A_j$ and $T^{k_{j}}$ acts 
on the fixed point $\pi^{-1}(v_{\ell-j+1})$ trivially. Note that $T^n$ 
acts on this bundle via the projection of $T^n \to T^{k_j}$ 
determined by \eqref{eq_ground_lattice}. 
Hence we have $T^n$-equivariant homeomorphisms 
$$X_{j}/X_{j-1}\cong \pi^{-1}(Q_{\ell-j+1})/ \pi^{-1}(Q_{\ell-j+1} \cap P_{\ell-j+2}) \cong \mathrm{Th}(\xi^{A_j}).
$$
\end{proof}

We refer to \cite[Proposition 4.4]{BNSS} for a relevant interpretation of a retraction sequence from the viewpoint of  $\mathbf{q}$-CW complexes.

Next corollary extends the result of \cite[Theorem 12.3.11]{CLS} from the category of 
simplicial toric varieties to the category of toric varieties over almost simple polytopes. 
\begin{corollary}\label{cor_tor_free_odd_van}
The ordinary cohomology with rational coefficients of a toric variety $X$ over an almost simple polytope vanishes in odd degrees, i.e., $H^{2k+1}(X;\QQ)=0$ for all $k$. 
\end{corollary}
\begin{proof}
The identifications \eqref{eq_difference_of_strata_q-disc} for each $j$ realizes a \emph{building sequence} defined in \cite[Definition 2.4]{BNSS} of $X$. Since each \eqref{eq_difference_of_strata_q-disc} is even dimensional, the result directly follows from \cite[Theorem 1.1]{BNSS}.
\end{proof}

\begin{proposition}
The $T^n$-equivariant stratification in \eqref{eq_T_equiv_stratification} 
satisfies assumptions {\rm (\textbf{A}1), (\textbf{A}2)} and {\rm (\textbf{A}3)} in Section \ref{sec:GenEqCohom}.  
\end{proposition}
\begin{proof}
Recall that $\pi^{-1}(v_{\ell-j+1})$ in \eqref{eq_orb_vector_bdl_over_pt} is a fixed point.
Observe that the total space of \eqref{eq_T_bundle} is a quotient of a 
$T^n$-representation on $\CC^{k_j}$ by a finite subgroup $A_j$ of $T^n$.
Therefore, $\CC^{k_j}$ can be decomposed into 1-dimensional representations
as $$ \CC^{k_j} \cong \CC(\alpha_1) \oplus \cdots \oplus \CC(\alpha_{k_j})$$ for some
characters $\alpha_s \colon T^n \to S^1$. Since each $\CC(\alpha_s)$ is 
invariant under $A_j$, we have
$$(\xi^{A_j} \colon \CC^{k_j}/A_j \to \pi^{-1}(v_{\ell-j+1})) \cong \bigoplus_{s=1}^{k_j} \big( \xi^{A_{j,s}} \colon \CC(\alpha_s)/A_{j,s} \to \pi^{-1}(v_{\ell-j+1})\big)$$
for some finite groups $A_{j,1}, \ldots, A_{j,k_j}$. This proves assumption (\textbf{A}1).

The quotient of the 1-dimensional representation $\CC(\alpha_s)$ by 
$T^n$-action is identical to $\RR_{\geq 0}$
which corresponds to an edge, say $e_s$, of $U_{\ell-j+1}$. Indeed,  since $\pi^{-1}(U_{\ell-j+1}) \to \pi^{-1}(v_{\ell-j+1})$ is the $\mathbf{q}$-disc bundle associated with $\xi^{A_j}\colon \CC^{k_j}/A_j \to \pi^{-1}(v_{\ell-j+1})$, one can see that two projections  
$\CC(\alpha_s)/A_{j,s} \to \pi^{-1}(v_{\ell-j+1})$ and $\pi^{-1}(e_s) \to \pi^{-1}(v_{\ell-j+1})$ are identical. Note that one can write the attaching map
$\phi_j$ explicitly by the proof of \cite[Theorem 4.1]{BSS}.
Therefore, the image of $\phi_j|_{S(\CC(\alpha_s))}$ is a vertex $v_{s}$ of $e_s$ which is opposite to $v_{\ell-j+1}$. A pictorial explanation is given in
Figure \ref{fig_attaching_map}. Considering 
$$f_{j,s} \colon \pi^{-1}(v_{\ell-j+1}) \to \pi^{-1}(v_s)$$ 
as a map between two fixed points, we conclude Assumption (\textbf{A}2). 
\begin{figure}
\begin{tikzpicture}[scale=0.8]
\draw[fill=green!20, thick] (0,0)--(3,0)--(4,0.7)--(3.5,2)--(0.5,2)--cycle;
\draw[thick] (3.5,2)--(3,0);
\draw[dotted, thick] (0,0)--(1,0.7)--(4,0.7);
\draw[dotted, thick] (1,0.7)--(0.5,2);

\node[fill, circle, inner sep=1pt] at (3.5,2) {};
\node[above] at (3.5,2) {\scriptsize$v_{\ell-j+1}$};

\node at (5, 1) {$\supset$};

\begin{scope}[xshift=180]
\draw[fill=green!20, thick] (0,0)--(1,.7)--(0.5,2)--cycle;
\draw[fill=green!20, thick] (0,0)--(3,0)--(4,.7)--(1,0.7)--cycle;
\draw[dashed, thick, ->] (3.5, 2)--(0.6,2);
\draw[dashed, thick, ->] (3.5, 2)--(3,0.1);
\draw[dashed, thick, ->] (3.5, 2)--(4,.8);

\draw[dotted, thick] (0.6,2)--(0.1, 0.1)--(3,0.1)--(4,0.8)--(1.1,0.8)--(0.6,2)--cycle;
\draw[dotted, thick] (0.1, 0.1)--(1.1,0.8); 

\node[fill, circle, inner sep=1pt] at (0.5,2) {};
\node[fill, circle, inner sep=1pt] at (3.5,2) {};
\node[fill, circle, inner sep=1pt] at (3,0) {};
\node[fill, circle, inner sep=1pt] at (4,0.7) {};

\node[above] at (3.5,2) {\scriptsize$v_{\ell-j+1}$};
\node[above] at (0.5,2) {\scriptsize$v_{1}$};
\node[below] at (3,0) {\scriptsize$v_{2}$};
\node[right] at (4,0.7) {\scriptsize$v_{3}$};

\node[above] at (2,2) {\scriptsize$f_{j,1}$};
\node[left] at (3.3,1.3) {\scriptsize$f_{j,2}$};
\node[right] at (3.8,1.5) {\scriptsize$f_{j,3}$};
\end{scope}
\end{tikzpicture}
\caption{An attaching map.}
\label{fig_attaching_map}
\end{figure}

Assumption (\textbf{A}3) follows from \cite[Lemma 5.2]{HHH}, as the vectors $\mu_1, \ldots, \mu_{k_j}$ defined by \eqref{eq_ground_lattice} are linearly independent.
\end{proof}

The following is an application of Proposition \ref{prop_image}
to the category of toric varieties over almost simple polytopes. 
\begin{proposition}\label{prop_GKM-description}
Let $X$ be a toric variety over an almost simple polytope with an orbifold $G$-equivariant stratification as in Theorem \ref{thm:q-cw_srtucture}. Let $E^\ast_{T^n}$ be a generalized $T^n$-equivariant cohomology theory discussed in Section \ref{sec:GenEqCohom}. Then, 
\begin{equation*}
E^*_{T^n}(X) = \Big \{(x_i) \in \prod_{1\leq i\leq \ell } E^*_{T^n}(pt) ~ \Big{|} ~ e_{T^n}(\xi^{A_{j,s}})~|~ x_{\ell-j+1} - f^*_{j,s}(x_{s}) ~\text{\emph{for}}~ s < \ell-j+1  \Big\}.
\end{equation*}
\end{proposition}

We note that $H^\ast_{T^n}(pt)$ and $K^\ast_{T^n}(pt)$ are 
isomorphic to the ring of polynomials and the ring of Laurant polynomials with 
$n$-variables, respectively. For $MU_{T^n}^\ast(pt)$, though its structure is unknown, 
it is referred as the ring of \emph{$T^n$-cobordism forms} in \cite{HHRW}. 

\section{Piecewise algebras and Applications}\label{sec:PPandAPPLICATIONS}
We begin this section with a summary of the concept of some piecewise algebras 
associated to a fan, studied in \cite[Section 4]{HHRW}. The authors apply those 
algebras to weighted projective spaces to get a description of generalized 
equivariant cohomology theories. Here, we generalize their several results to a wider class of singular toric varieties discussed in Section \ref{sec:toric_variety}. 

Recall that if $\sigma$ is a cone in a fan $\Sigma$, then all of the
faces of $\sigma$ belong to $\Sigma$. 
This leads us to form a small category $\ts{cat}(\Sigma)$ whose objects 
are elements of $\Sigma$ and morphisms are face inclusions. The zero cone $\{0\}$
is the initial object of this category. 

Let $\Sigma$ be an $n$-dimensional rational fan in $\RR^n$, 
namely, one-dimensional cones are generated by rational vectors in $\RR^n$.  
Here we may identify $\RR^n$ with the Lie algebra of $T^n$. Given a
$k(\leq n)$-dimensional cone $\sigma\in \Sigma$, we consider a 
subtorus $T_\sigma$ generated by primitive vectors spanning 
$1$-dimensional cones in $\sigma$.
For the category $T^n$-\ts{top} of $T^n$-spaces, we define a diagram 
\begin{equation}\label{eq_V-Diagram}
\mathcal{V}\colon \ts{cat}(\Sigma) \to T^n\ts{-top}
\end{equation}
by $\mathcal{V}(\sigma):=T^n/T_\sigma$ 
and $\mathcal{V}(\sigma \subseteq \tau)=(T^n/T_\sigma \twoheadrightarrow T^n/T_\tau)$, 
where the projection $T^n/T_\sigma \twoheadrightarrow T^n/T_\tau$ is induced from 
the natural inclusion $T_\sigma \subseteq T_\tau$. Then, 
the toric variety $X_\Sigma$ associated to $\Sigma$ is homotopy equivalent to 
the homotopy colimit ${\rm hocolim}\mathcal{V}$ of \eqref{eq_V-Diagram}, 
see for instance \cite[Section 4]{HHRW} as well as \cite{Fr, WZZ}. 
Next, regarding $E_{T^n}^\ast$ as a functor from $T^n\ts{-top}$ to the 
category $\ts{gcalg}_E$ of graded commutative $E_{T^n}^\ast$-algebras, 
we consider the composition 
\begin{equation}\label{eq_evaluation_map}
\mathcal{EV}\colon \ts{cat}(\Sigma) \xrightarrow{\mathcal{V}} T^n\ts{-top} \xrightarrow{E^\ast_{T^n}} \ts{gcalg}_{E},
\end{equation}
which leads us to the following definition.
\begin{definition}\cite[Definition 4.6]{HHRW}\label{piec_alg}
Let $\Sigma$ be an $n$-dimensional  rational fan in $\RR^n$. We call $\lim \mathcal{EV}$
the \emph{piecewise algebra} over $E_{T^n}^\ast$. 
\end{definition}

We note that the object $\mathcal{EV}(\sigma)$ can be calculated explicitly as follows. 
The natural action of $T^n$ on $\mathcal{V}(\sigma)=T^n/T_\sigma$ yields 
a $T^n$-representation $\eta_\sigma$ on which $T_\sigma$ acts trivially. 
Since $T^n$ is abelian, $\eta_\sigma$ can be decomposed into 1-dimensional 
representations, say $\eta_\sigma \cong \bigoplus_{i=1}^{n-k}\eta_\sigma(i)$, where $k=\dim T_\sigma$.  
We denote by $S^1_{\eta_{\sigma}(i)}$ the corresponding circle for each $i=1, \dots, n-k$. 
The inclusion of $S^1_{\eta_{\sigma}(i)}$ into the unit disc $D_{\eta_{\sigma}(i)}$
gives an equivariant cofiber sequence 
\begin{equation}\label{eq_cofib_1_dimen}
S^1_{\eta_{\sigma}(i)} \hookrightarrow D_{\eta_{\sigma}(i)} \to D_{\eta_{\sigma}(i)}/ S^1_{\eta_{\sigma}(i)}.
\end{equation}
Regarding each term of \eqref{eq_cofib_1_dimen} as an $S^1$-bundle, disc bundle over a point 
and the associated Thom space,  respectively, we may consider the equivariant Euler class 
$e_{T^n}(\eta_\sigma(i)) \in E^*_{T^n}$ for each $i=1, \dots, n-k$. 
%
%
\begin{proposition}\cite[Section 4, (4.11)]{HHRW}
$$\mathcal{EV}(\sigma) \cong E_{T^n}^\ast(T^n/T_\sigma) \cong E^\ast_{T^n}(pt)/(e_{T^n}\big(\eta_\sigma(1)), \dots, e_{T^n}(\eta_\sigma(n-k))\big).$$
\end{proposition}

The proof of the following theorem is almost same as the proof of \cite[Theorem 5.5]{HHRW}
with very few modifications in notation. To be more precise, one needs to replace the
equivariant Euler classes \cite[(3.8)]{HHRW} corresponding to the filtration of a divisive weighted projective space by the equivariant Euler classes 
for simple orbifold bundles defined in Section \ref{sec:GenEqCohom}. 

\begin{theorem}\label{thm_main_E*_T=PP}
Let $X_P$ be a toric variety over an almost simple polytope $P$ and 
$\Sigma_P$ the normal fan of $P$. Let $E_{T^n}^\ast$ be a $T^n$-equivariant cohomology theory discussed in Section \ref{sec:GenEqCohom}. Then, $E_{T^n}^\ast(X_P)$ is isomorphic to the piecewise algebra $\lim \mathcal{EV}$, as $E_{T^n}(pt)$-algebras. 
\end{theorem}

\begin{remark}
Compact symplectic toric orbifolds are toric varieties with fans defined by their moment polytope which are simple, see \cite[Section 9]{LeTo}. Therefore they satisfy the hypotheses of Theorem \ref{thm:q-cw_srtucture} and Theorem \ref{thm_main_E*_T=PP}. 
\end{remark}



\section{Divisive toric varieties}\label{sec:divisive_toric}
In this section, we introduce the notion of {\it divisive toric varieties} motivated by divisive weighted projective spaces \cite{BFR}. They are singular toric varieties which may have singularities beyond orbifold singularities, whose generalized equivariant cohomologies can be obtained over integers. We follow the same arguments as we 
discussed in Section \ref{sec:PPandAPPLICATIONS}. Here, the cohomology theory $E^\ast_{T^n}$ in this section can also be $MU_{T^n}^\ast$ without taking tensor with $\QQ$.

\begin{definition}
Let $X$ be a toric variety satisfying the hypothesis of Theorem 
\ref{thm:q-cw_srtucture}. If the finite groups $A_j$'s in \eqref{eq_T_bundle}
are trivial, then we call $X$ a \emph{divisive toric variety}. 
\end{definition}

\begin{example}\label{ex_GZ_divisive}
Recall the $3$-dimensional Gelfand--Zetlin polytope $P$ described in 
Example \ref{ex_GZ_polytope}. The outward normal vectors of 
facets intersecting $v_1=(0,2,2)$ in $P_1$ of Figure \ref{fig_ret_seq_of_GZ_polytope}
are $(-1,0,0)$, $(0,-1,1)$ and $(0,1,0)$, which form an integral basis of $\ZZ^3$. 
See Figure \ref{fig_3-dim_GC} for the vertices and primitive outward normal vectors of $P$. 
Hence, the finite group $A_7$ defined in \eqref{eq_finite_gp_in_filtration} is
trivial. To compute $A_6$, we consider the facet given by $\{x=1\}$ whose 
primitive outward normal vector is $(1,0,0)$. In this case, the map 
\eqref{eq_ground_lattice} yields the projection $T^3\twoheadrightarrow T^2$
onto the last two coordinates. Hence, 
$A_6=\ker(\rho \colon T^2 \to T^2)$, where $\rho(t_1, t_2)=(t_1^{-1}t_2, t_1)$, 
which is trivial. We refer to \cite[Proposition 4.3]{BSS} for the general 
statement about this computation. 
Finally, one can conclude by similar computations for the other 
vertices $v_2, \dots, v_6$ that the associated toric variety $X_P$ is divisive. 
\end{example}
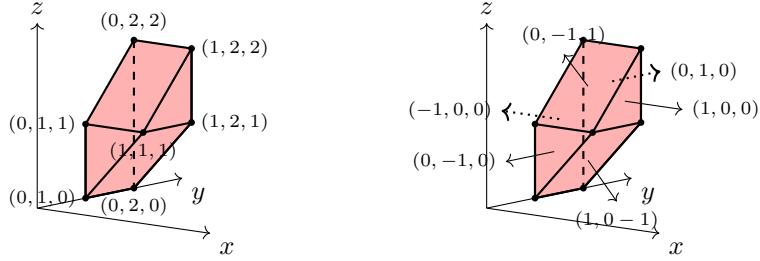
\begin{figure}
\tdplotsetmaincoords{80}{40}
\begin{tikzpicture}[tdplot_main_coords]
\draw[->] (0,0,0) -- (3,0,0) node[anchor=north west]{$x$};
\draw[->] (0,0,0) -- (0,3,0) node[anchor=north west]{$y$};
\draw[->] (0,0,0) -- (0,0,2.5) node[anchor=south]{$z$};
\draw[thick, fill=red!30] (0,1,0)--(0,2,0)--(1,2,1)--(1,2,2)--(0,2,2)--(0,1,1)--cycle;
\draw[thick, dashed] (0,2,2)--(0,2,0);
\draw[thick] (0,1,1)--(0,1,0)--(0,2,0)--(1,2,1)--(1,2,2);
\draw[thick] (1,1,1)--(0,1,0);
\draw[thick] (1,1,1)--(0,1,1);
\draw[thick] (1,1,1)--(1,2,1);
\draw[thick] (1,1,1)--(1,2,2);

\node[fill, circle, inner sep=0.9pt] at (0,2,2) {};
\node[fill, circle, inner sep=.9pt] at (0,1,0) {};
\node[fill, circle, inner sep=.9pt] at (0,1,1) {};
\node[fill, circle, inner sep=.9pt] at (1,1,1) {};
\node[fill, circle, inner sep=.9pt] at (1,2,1) {};
\node[fill, circle, inner sep=.9pt] at (1,2,2) {};
\node[fill, circle, inner sep=.9pt] at (0,2,0) {};

\node[above] at (0,2,2) {\scriptsize $(0,2,2)$};
\node[right] at (1,2,2) {\scriptsize $(1,2,2)$};
\node[right] at (1,2,1) {\scriptsize $(1,2,1)$};
\node[below] at (0,2,0) {\scriptsize $(0,2,0)$};
\node[left] at (0,1,0) {\scriptsize $(0,1,0)$};
\node[left] at (0,1,1) {\scriptsize $(0,1,1)$};
\node[below] at (1,1,1) {\scriptsize $(1,1,1)$};

\begin{scope}[xshift=170]
\draw[->] (0,0,0) -- (3,0,0) node[anchor=north west]{$x$};
\draw[->] (0,0,0) -- (0,3,0) node[anchor=north west]{$y$};
\draw[->] (0,0,0) -- (0,0,2.5) node[anchor=south]{$z$};
\draw[thick, fill=red!30] (0,1,0)--(0,2,0)--(1,2,1)--(1,2,2)--(0,2,2)--(0,1,1)--cycle;
\draw[thick, dashed] (0,2,2)--(0,2,0);
\draw[thick] (0,1,1)--(0,1,0)--(0,2,0)--(1,2,1)--(1,2,2);
\draw[thick] (1,1,1)--(0,1,0);
\draw[thick] (1,1,1)--(0,1,1);
\draw[thick] (1,1,1)--(1,2,1);
\draw[thick] (1,1,1)--(1,2,2);

\node[fill, circle, inner sep=0.9pt] at (0,2,2) {};
\node[fill, circle, inner sep=.9pt] at (0,1,0) {};
\node[fill, circle, inner sep=.9pt] at (0,1,1) {};
\node[fill, circle, inner sep=.9pt] at (1,1,1) {};
\node[fill, circle, inner sep=.9pt] at (1,2,1) {};
\node[fill, circle, inner sep=.9pt] at (1,2,2) {};
\node[fill, circle, inner sep=.9pt] at (0,2,0) {};

\draw[->,] (1,5/3, 4/3)--(2,5/3, 4/3); \node[right] at (2,5/3, 4/3) {\scriptsize $(1,0,0)$};
\draw[->] (1/2, 3/2, 3/2)--(1/2, 1, 2); \node[above] at (1/2, 1, 2) {\scriptsize $(0, -1,1)$};
\draw[->] (1/3, 1, 2/3)--(1/3, 0, 2/3); \node[left] at (1/3, 0, 2/3) {\scriptsize $(0,-1,0)$};
\draw[->] (1/2, 3/2, 1/2)--(1, 3/2, 0); \node[below] at (1, 3/2, 0) {\scriptsize $(1, 0 -1)$};
\draw[->,dotted, thick] (1/2, 2, 3/2)--(1/2, 3, 3/2); \node[right] at (1/2, 3, 3/2) {\scriptsize $(0,1,0)$};
\draw[->,dotted, thick] (0, 3/2, 1)--(-1, 3/2, 1); \node[left] at (-1, 3/2, 1) {\scriptsize $(-1, 0, 0)$};

\end{scope}
\end{tikzpicture}
\caption{A 3-dimensional Gelfand--Zetlin Polytope.}
\label{fig_3-dim_GC}
\end{figure}

The following proposition is straightforward from Corollary \ref{cor_tor_free_odd_van} and the definition of a divisive toric variety. 
\begin{proposition}
The ordinary cohomology with integer coefficients of a divisive toric variety $X$ is torsion free and vanishes in vanishes in odd degrees.
\end{proposition}

Note that \cite{HHRW} defines the piecewise algebras of a fan with integers. To be more precise, the map  $\mathcal{EV}$ in \eqref{eq_evaluation_map} is a composition of $\mathcal{V}$ and the $T^n$-equivariant cohomology theories without taking tensor with $\QQ$. In particular, we denote the associated piecewise algebras by 
\begin{itemize}
\item $PP[\Sigma]$, the algebra of piecewise polynomials if $E_{T^n}^\ast = H^\ast_{T^n}$;
\item $PL[\Sigma]$, the algebra of piecewise Laurant polynomials if $E_{T^n}^\ast = K^\ast_{T^n}$;
\item $PC[\Sigma]$, the algebra of piecewise $T^n$-cobordism forms if $E_{T^n}^\ast = MU^\ast_{T^n}$. 
\end{itemize}
If a toric variety is divisive, then it is 
equipped with a $T^n$-equivariant stratification in the sense of
\cite[Section 2]{HHH}. So one can apply their results to divisive 
toric varieties, which yields the following proposition over integers.

\begin{proposition}\label{prop_divisive_PP}
Let $X$ be a divisive toric variety over an almost simple polytope $P$ 
and $\Sigma_P$ the normal fan of $P$. Then, 
\begin{enumerate}
\item $H^\ast_{T^n}(X_P; \ZZ)$ is isomorphic to $PP[\Sigma_P]$ as an
 $H^\ast_{T^n}(pt; \ZZ)$-algebra;
\item $K^\ast_{T^n}(X_P; \ZZ)$ is isomorphic to $PL[\Sigma_P]$ as a
$K^\ast_{T^n}(pt; \ZZ)$-algebra;
\item $MU^\ast_{T^n}(X_P; \ZZ)$ is isomorphic to $PC[\Sigma_P]$ as an 
$MU^\ast_{T^n}(pt; \ZZ)$-algebra.
\end{enumerate}
\end{proposition}

To exhibit an example of piecewise algebra, we revisit the $3$-dimensional Gelfand--Zetlin polytope discussed in Example \ref{ex_GZ_polytope} and Example \ref{ex_GZ_divisive}.

\begin{example}\label{ex_GZ_toric_PP}
Let $P$ be the $3$-dimensional Gelfand--Zetlin polytope and $\Sigma_P$ its normal fan. 
For each face $Q$ in $P$ of dimension $k$ $(0 \leq k\leq 3)$, we denote by $\sigma_Q$ the associated $(3-k)$-dimensional cone in $\Sigma_P$, i.e., $\sigma_Q$ is the cone generated by normal vectors of facets intersecting the relative interior of $Q$. For example, 
$\sigma_{(1,1,1)}$ is the cone generated by $(0,-1,1), (0,-1,0), (1,0,-1)$ and $(1,0,0)$. Particularly when $E_{T^3}^\ast=H_{T^3}^\ast$,
$$\mathcal{EV}(\sigma_v)\cong H_{T^3}^\ast(pt) \cong \ZZ[u_1, u_2, u_3]$$
for each vertex $v$ of $P$. 
Hence, the ring $PP[\Sigma_P]$ of piecewise polynomials with rational coefficients is the set of tuples 
$$(f_{\sigma_v})_{v\in  V(P)} \in \bigoplus_{v\in V(P)}\ZZ[u_1, u_2, u_3]$$ 
such that $f_{\sigma_v}|_{\sigma_e} =f_{\sigma_w}|_{\sigma_e}$ whenever $v$ and $w$ are connected by an edge $e$. 
Here, we list some of its elements as follows:
\begin{table}[h]
\begin{tabular}{ccccccc}
$\sigma_{(1,1,1)}$ & $\sigma_{(0,1,0)}$ & $\sigma_{(0,2,0)}$ & $\sigma_{(0,1,1)}$ & $\sigma_{(1,2,1)}$ & $\sigma_{(1,2,2)}$ & $\sigma_{(0,2,2)}$ \\ \hline \hline
1 & 1 & 1 & 1 & 1 & 1 & 1 \\ \hline
0 & 0 & $u_2$ & 0 & $u_2$ & $u_2+u_3$ & $u_2+u_3$ \\ \hline
$u_1+u_3$ & 0 & 0 & $u_3$ & $u_1+u_3$ & $u_1+u_3$ & $u_3$ \\ \hline
0 & 0 & 0 & 0 & $u_2(u_1+u_3)$ & $u_1(u_2+u_3)$ & 0 \\ \hline
0 & 0 & 0 & 0 & 0 & $u_3(u_2+u_3)$ & $u_3(u_2+u_3)$ \\ \hline
0 & 0 & 0 & 0 & 0 &  $u_1u_3(u_2+u_3)$ & 0 \\ \hline
\end{tabular}
\caption{Some elements in $PP[\Sigma_P]$.}
\label{tab_elements_of_PP[GZ]}
\end{table}

\noindent See Figure \ref{Fig_PPonQ} for the description of those elements on the original polytope $P$. 
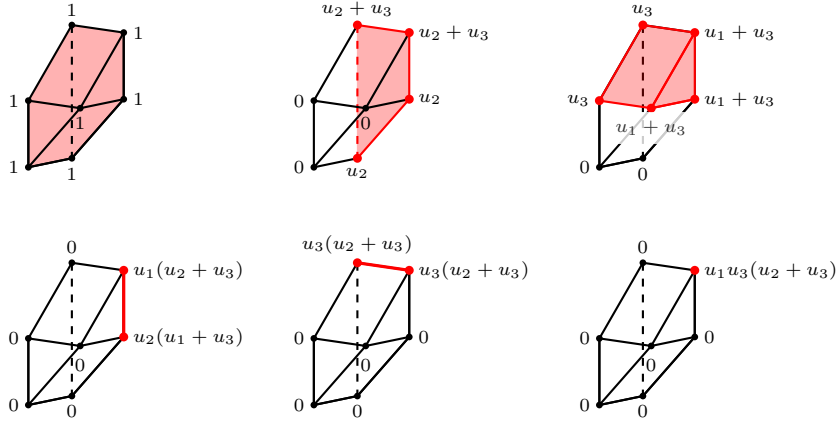
\begin{figure}[h]
\begin{tikzpicture}
\begin{scope}[scale=0.9, yscale=0.9]
\tdplotsetmaincoords{80}{40}
\begin{scope}[tdplot_main_coords]
\draw[thick, fill=red!30] (0,1,0)--(0,2,0)--(1,2,1)--(1,2,2)--(0,2,2)--(0,1,1)--cycle;
\draw[thick, dashed] (0,2,2)--(0,2,0);
\draw[thick] (0,1,1)--(0,1,0)--(0,2,0)--(1,2,1)--(1,2,2);
\draw[thick] (1,1,1)--(0,1,0);
\draw[thick] (1,1,1)--(0,1,1);
\draw[thick] (1,1,1)--(1,2,1);
\draw[thick] (1,1,1)--(1,2,2);

\node[fill, circle, inner sep=0.9pt] at (0,2,2) {};
\node[fill, circle, inner sep=.9pt] at (0,1,0) {};
\node[fill, circle, inner sep=.9pt] at (0,1,1) {};
\node[fill, circle, inner sep=.9pt] at (1,1,1) {};
\node[fill, circle, inner sep=.9pt] at (1,2,1) {};
\node[fill, circle, inner sep=.9pt] at (1,2,2) {};
\node[fill, circle, inner sep=.9pt] at (0,2,0) {};

\node[above] at (0,2,2) {\scriptsize $1$};
\node[right] at (1,2,2) {\scriptsize $1$};
\node[right] at (1,2,1) {\scriptsize $1$};
\node[below] at (0,2,0) {\scriptsize $1$};
\node[left] at (0,1,0) {\scriptsize $1$};
\node[left] at (0,1,1) {\scriptsize $1$};
\node[below] at (1,1,1) {\scriptsize $1$};
\end{scope}

\begin{scope}[xshift=120, tdplot_main_coords] 
\draw[fill=red!30, red!30] (0,2,0)--(1,2,1)--(1,2,2)--(0,2,2)--cycle;
\draw[thick, red] (0,2,0)--(1,2,1)--(1,2,2)--(0,2,2);
\draw[thick, red, dashed] (0,2,2)--(0,2,0);
\draw[thick] (0,1,0)--(0,2,0);
\draw[thick] (0,1,0)--(0,1,1);
\draw[thick] (0,1,1)--(0,2,2);
\draw[thick] (1,1,1)--(0,1,0);
\draw[thick] (1,1,1)--(0,1,1);
\draw[thick] (1,1,1)--(1,2,1);
\draw[thick] (1,1,1)--(1,2,2);

\node[fill=red, circle, inner sep=1.2pt] at (0,2,2) {};
\node[fill, circle, inner sep=.9pt] at (0,1,0) {};
\node[fill, circle, inner sep=.9pt] at (0,1,1) {};
\node[fill, circle, inner sep=.9pt] at (1,1,1) {};
\node[fill=red, circle, inner sep=1.2pt] at (1,2,1) {};
\node[fill=red, circle, inner sep=1.2pt] at (1,2,2) {};
\node[fill=red, circle, inner sep=1.2pt] at (0,2,0) {};

\node[above] at (0,2,2) {\scriptsize $u_2+u_3$};
\node[right] at (1,2,2) {\scriptsize $u_2+u_3$};
\node[right] at (1,2,1) {\scriptsize $u_2$};
\node[below] at (0,2,0) {\scriptsize $u_2$};
\node[left] at (0,1,0) {\scriptsize $0$};
\node[left] at (0,1,1) {\scriptsize $0$};
\node[below] at (1,1,1) {\scriptsize $0$};
\end{scope}

\begin{scope}[xshift=240, tdplot_main_coords] 
\draw[fill=red!30, red!30] (1,1,1)--(1,2,1)--(1,2,2)--cycle;
\draw[thick] (0,1,0)--(0,2,0)--(1,2,1)--(1,2,2)--(0,2,2)--(0,1,1)--cycle;
\draw[thick, dashed] (0,2,2)--(0,2,0);
\draw[thick] (0,1,1)--(0,1,0)--(0,2,0)--(1,2,1)--(1,2,2);

\draw[fill=red!30, red!30] (1,1,1)--(1,2,1)--(1,2,2)--cycle;
\draw[fill=red, red, opacity=0.3] (1,1,1)--(1,2,2)--(0,2,2)--(0,1,1)--cycle;
\draw[thick, red] (1,1,1)--(1,2,1)--(1,2,2)--cycle;
\draw[thick] (1,1,1)--(0,1,0);
\draw[thick, red] (1,1,1)--(0,1,1);
\draw[thick, red] (1,1,1)--(1,2,1);
\draw[thick, red] (1,1,1)--(1,2,2);
\draw[thick, red] (0,1,1)--(0,2,2)--(1,2,2);

\node[fill=red, circle, inner sep=1.2pt] at (0,2,2) {};
\node[fill, circle, inner sep=.9pt] at (0,1,0) {};
\node[fill=red, circle, inner sep=1.2pt] at (0,1,1) {};
\node[fill=red, circle, inner sep=1.2pt] at (1,1,1) {};
\node[fill=red, circle, inner sep=1.2pt] at (1,2,1) {};
\node[fill=red, circle, inner sep=1.2pt] at (1,2,2) {};
\node[fill, circle, inner sep=.9pt] at (0,2,0) {};

\node[above] at (0,2,2) {\scriptsize $u_3$};
\node[right] at (1,2,2) {\scriptsize $u_1+u_3$};
\node[right] at (1,2,1) {\scriptsize $u_1+u_3$};
\node[below] at (0,2,0) {\scriptsize $0$};
\node[left] at (0,1,0) {\scriptsize $0$};
\node[left] at (0,1,1) {\scriptsize $u_3$};
\node[fill=white, opacity=0.8, below] at (1,1,0.95) {\scriptsize $u_1+u_3$};
\end{scope}

\begin{scope}[yshift=-90, tdplot_main_coords]
\draw[thick] (0,1,0)--(0,2,0)--(1,2,1)--(1,2,2)--(0,2,2)--(0,1,1)--cycle;
\draw[thick, dashed] (0,2,2)--(0,2,0);
\draw[thick] (0,1,1)--(0,1,0)--(0,2,0)--(1,2,1)--(1,2,2);
\draw[thick] (1,1,1)--(0,1,0);
\draw[thick] (1,1,1)--(0,1,1);
\draw[thick] (1,1,1)--(1,2,1);
\draw[thick] (1,1,1)--(1,2,2);
\draw[very thick, red] (1,2,1)--(1,2,2);

\node[fill, circle, inner sep=0.9pt] at (0,2,2) {};
\node[fill, circle, inner sep=.9pt] at (0,1,0) {};
\node[fill, circle, inner sep=.9pt] at (0,1,1) {};
\node[fill, circle, inner sep=.9pt] at (1,1,1) {};
\node[fill=red, circle, inner sep=1.2pt] at (1,2,1) {};
\node[fill=red, circle, inner sep=1.2pt] at (1,2,2) {};
\node[fill, circle, inner sep=.9pt] at (0,2,0) {};

\node[above] at (0,2,2) {\scriptsize $0$};
\node[right] at (1,2,2) {\scriptsize $u_1(u_2+u_3)$};
\node[right] at (1,2,1) {\scriptsize $u_2(u_1+u_3)$};
\node[below] at (0,2,0) {\scriptsize $0$};
\node[left] at (0,1,0) {\scriptsize $0$};
\node[left] at (0,1,1) {\scriptsize $0$};
\node[below] at (1,1,0.95) {\scriptsize $0$};

\end{scope}

\begin{scope}[xshift=120, yshift=-90,  tdplot_main_coords]
\draw[thick] (0,1,0)--(0,2,0)--(1,2,1)--(1,2,2)--(0,2,2)--(0,1,1)--cycle;
\draw[thick, dashed] (0,2,2)--(0,2,0);
\draw[thick] (0,1,1)--(0,1,0)--(0,2,0)--(1,2,1)--(1,2,2);
\draw[thick] (1,1,1)--(0,1,0);
\draw[thick] (1,1,1)--(0,1,1);
\draw[thick] (1,1,1)--(1,2,1);
\draw[thick] (1,1,1)--(1,2,2);
\draw[very thick, red] (1,2,2)--(0,2,2);

\node[fill=red, circle, inner sep=1.2pt] at (0,2,2) {};
\node[fill, circle, inner sep=.9pt] at (0,1,0) {};
\node[fill, circle, inner sep=.9pt] at (0,1,1) {};
\node[fill, circle, inner sep=.9pt] at (1,1,1) {};
\node[fill, circle, inner sep=.9pt] at (1,2,1) {};
\node[fill=red, circle, inner sep=1.2pt] at (1,2,2) {};
\node[fill, circle, inner sep=.9pt] at (0,2,0) {};

\node[above] at (0,2,2) {\scriptsize $u_3(u_2+u_3)$};
\node[right] at (1,2,2) {\scriptsize $u_3(u_2+u_3)$};
\node[right] at (1,2,1) {\scriptsize $0$};
\node[below] at (0,2,0) {\scriptsize $0$};
\node[left] at (0,1,0) {\scriptsize $0$};
\node[left] at (0,1,1) {\scriptsize $0$};
\node[below] at (1,1,0.95) {\scriptsize $0$};

\end{scope}

\begin{scope}[xshift=240, yshift=-90,  tdplot_main_coords]
\draw[thick] (0,1,0)--(0,2,0)--(1,2,1)--(1,2,2)--(0,2,2)--(0,1,1)--cycle;
\draw[thick, dashed] (0,2,2)--(0,2,0);
\draw[thick] (0,1,1)--(0,1,0)--(0,2,0)--(1,2,1)--(1,2,2);
\draw[thick] (1,1,1)--(0,1,0);
\draw[thick] (1,1,1)--(0,1,1);
\draw[thick] (1,1,1)--(1,2,1);
\draw[thick] (1,1,1)--(1,2,2);

\node[fill, circle, inner sep=0.9pt] at (0,2,2) {};
\node[fill, circle, inner sep=.9pt] at (0,1,0) {};
\node[fill, circle, inner sep=.9pt] at (0,1,1) {};
\node[fill, circle, inner sep=.9pt] at (1,1,1) {};
\node[fill, circle, inner sep=.9pt] at (1,2,1) {};
\node[fill=red, circle, inner sep=1.2pt] at (1,2,2) {};
\node[fill, circle, inner sep=.9pt] at (0,2,0) {};

\node[above] at (0,2,2) {\scriptsize $0$};
\node[right] at (1,2,2) {\scriptsize $u_1u_3(u_2+u_3)$};
\node[right] at (1,2,1) {\scriptsize $0$};
\node[below] at (0,2,0) {\scriptsize $0$};
\node[left] at (0,1,0) {\scriptsize $0$};
\node[left] at (0,1,1) {\scriptsize $0$};
\node[below] at (1,1,0.95) {\scriptsize $0$};

\end{scope}
\end{scope}
\end{tikzpicture}
\caption{Dual description of Table \ref{tab_elements_of_PP[GZ]}.}
\label{Fig_PPonQ}
\end{figure}
\end{example}

\begin{remark}
The first element in Table \ref{tab_elements_of_PP[GZ]} or Figure \ref{Fig_PPonQ} is the multiplicative identity of $PP[\Sigma_P]$. The faces or the unions of faces in Figure \ref{Fig_PPonQ} whose vertices have nontrivial elements in $\ZZ[u_1, u_2, u_3]$ are \emph{dual Kogan faces} \cite{KoMi} and polynomials are related to Thom classes defined in \cite{MMP, MP}. 
\end{remark}

%
%
%
%
%
%

%
%

\end{document}